\documentclass[11pt]{article}
\usepackage[dvips]{epsfig}
\usepackage{amsmath,amssymb,euscript,graphicx,amsfonts,verbatim}
\usepackage{enumerate,color}
\usepackage{graphicx}
\usepackage{hyperref}

\usepackage{graphicx}

\newtheorem{theorem}{Theorem}[section]
\newtheorem{proposition}[theorem]{Proposition}
\newtheorem{lemma}[theorem]{Lemma}
\newtheorem{corollary}[theorem]{Corollary}

\newtheorem{problem}[theorem]{Problem}
\newtheorem{example}[theorem]{Example}

\newenvironment{proof}{{\noindent \sc Proof. }}{\hfill $\Qed$\\}

\newcommand{\la}{\langle}
\newcommand{\ra}{\rangle}

\def\SL{\hbox{\rm SL}}
\def\PGL{\hbox{\rm PGL}}
\def\PSL{\hbox{\rm PSL}}

\def\GL{\hbox{\rm GL}}

\newcommand{\Tr}{\hbox{{\rm Tr}}}
\newcommand{\Qed}{\rule{2.5mm}{3mm}}
\newcommand{\Aut}{\hbox{{\rm Aut}}\,}

\newcommand{\Cay}{\hbox{{\rm Cay}}}

\renewcommand{\S}{\mathcal{S}}

\newcommand{\Sym}{\hbox{{\rm Sym}}}

\newcommand{\ZZ}{\mathbb{Z}}

\newcommand{\CC}{\mathbb{C}}

\newcommand{\mm}[4]{
\left[ \begin{array}{cc}
#1 & #2 \\
#3 & #4
\end{array} \right]
}

\newcounter{case}

\renewcommand{\thecase}{\arabic{case}}

\newcounter{subcase}

\numberwithin{subcase}{case}

\begin{document}


\begin{center}
{\bf\large INTERSECTION DENSITY OF TRANSITIVE GROUPS\\
WITH CYCLIC POINT STABILIZERS} \\ [+4ex]
Ademir Hujdurovi\'c{\small$^{a,b,*,}$}\footnotemark,  \ 
Istv\'an Kov\'acs {\small$^{a,b,}$}\footnotemark,  \ 
Klavdija Kutnar{\small$^{a,b,}$}\footnotemark  \ 
\ and
\addtocounter{footnote}{0} 
Dragan Maru\v si\v c{\small$^{a, b, c,}$}\footnotemark 
\\ [+2ex]
{\it \small 
$^a$University of Primorska, UP IAM, Muzejski trg 2, 6000 Koper, Slovenia\\
$^b$University of Primorska, UP FAMNIT, Glagolja\v ska 8, 6000 Koper, Slovenia\\
$^c$IMFM, Jadranska 19, 1000 Ljubljana, Slovenia}
\end{center}

\addtocounter{footnote}{-3}
\footnotetext{The work of Ademir Hujdurovi\'c  is supported in part by the Slovenian Research Agency (research program P1-0404 and research projects  J1-7051, N1-0062, J1-9110, J1-1690, J1-1694, J1-1695, N1-0140, N1-0159, J1-2451, N1-0208).}
\addtocounter{footnote}{1}
\footnotetext{The work of Istv\'an Kov\'acs  is supported in part by the Slovenian Research Agency (research program P1-0285 and research projects N1-0062, J1-9108, J1-1695, N1-0140, J1-2451, N1-0208, J3-3001).}
\addtocounter{footnote}{1}
\footnotetext{The work of Klavdija Kutnar  is supported in part by the Slovenian Research Agency (research program P1-0285 and research projects N1-0062, J1-9110, J1-9186, J1-1695, J1-1715, N1-0140, J1-2451, J1-2481, N1-0209, J3-3001).}
\addtocounter{footnote}{1}
\footnotetext{
The work of Dragan Maru\v si\v c is supported in part by the Slovenian Research Agency (I0-0035, research program P1-0285 and research projects N1-0062, J1-9108, J1-1695, N1-0140, J1-2451, J3-3001).

~*Corresponding author e-mail:~ademir.hujdurovic@upr.si}

\begin{abstract}
For a  permutation group $G$  acting on a set $V$,
a subset $\mathcal{F}$ of $G$ is said to be an {\em intersecting set}
if for every pair of elements $g,h\in \mathcal{F}$ 
there exists $v \in V$ such that
$g(v) = h(v)$.
The {\em intersection density} $\rho(G)$ of a transitive permutation
group $G$ is the maximum value of the quotient $|\mathcal{F}|/|G_v|$ 
where $G_v$ is a stabilizer of a point $v\in V$ and 
$\mathcal{F}$ runs over all intersecting sets in $G$. 
 If $G_v$ is a largest intersecting set in $G$ then
$G$ is said to have  the {\em Erd\H{o}s-Ko-Rado (EKR)-property}. 
This paper is devoted to the study of transitive permutation groups,
with point stabilizers of  prime order with a special emphasis
given to orders 2 and 3, which do not have the EKR-property.
Among other, 
constructions of infinite family
of transitive permutation groups 
having point stabilizer of order $3$  with 
intersection density $4/3$ 
and of infinite families 
of transitive permutation groups 
having point stabilizer of order $3$
with arbitrarily large intersection density
are given.
\end{abstract}

\begin{quotation}
\noindent {\em Keywords:} 
intersection density, transitive permutation group, derangement graph.
\end{quotation}

\begin{quotation}
\noindent 
{\em Math. Subj. Class.:} 05C25, 20B25.
\end{quotation}


\section{Introductory remarks}
\label{sec:intro}
\noindent

The Erd\H os-Ko-Rado theorem \cite{EKR},
one of the central results in extremal combinatorics, which   
gives a bound on the size of a family of intersecting 
$k$-subsets of a set and classifies the families satisfying the bound,
has been extended in various ways.
This paper is concerned with an extension of 
this theorem to the ambient of 
transitive permutation groups.

For a finite set $V$ let $\Sym(V)$ denote the corresponding symmetric group, and if 
$|V|=n$,  the notation $S_n$ will be adopted.
Given a permutation group $G \leq \Sym(V)$, a subset $\mathcal{F}$ of $G$ is called \emph{intersecting} if, for any two  $g, h \in \mathcal{F}$, there exists $v \in V$ 
such that $g(v)=h(v)$.
The {\em intersection density} $\rho({\cal F})$ of the intersecting set ${\cal F}$
is defined to be the quotient
$$
\rho({\cal F})=\frac{|{\cal F}|}{\max\{|G_v|\colon v\in V\}},
$$
where $G_v$ is the point stabilizer of $v\in V$,
and the {\em intersection density} $\rho(G)$ (see \cite{LSP})
of a group $G$,
is the maximum value of
$\rho({\cal F})$ where ${\cal F}$ runs over all intersecting sets in
$G$. 
Clearly, every coset
$gG_v$, $v\in V$ and $g\in G$,
is an intersecting set, referred to as
a {\em canonical} intersecting set.
Consequently,
$\rho(G)\ge 1$. 
We say that $G$ has the \emph{Erd\H{o}s-Ko-Rado property} (in short {\em EKR-property}), if the size of a maximum intersecting set is equal to the order of the largest point stabilizer, and is said to have 
the \emph{strict Erd\H{o}s-Ko-Rado property} (in short {\em strict-EKR-property})
 if every maximum intersecting set of $G$ is a coset of a point stabilizer. 
It is clear that strict-EKR-property implies EKR-property, but the converse does not hold.  
In particular,
for a transitive group $G$ it follows that
$\rho(G)= 1$
if and only if
the maximum cardinality of an intersecting set is $|G|/|V|$.
Note that if $\mathcal{F}$ is an intersecting set of $G$ and
$f\in \mathcal{F}$ then  $f^{-1}\mathcal{F}$ is an intersecting
of $G$ (see
\cite[Proposition~2.3]{HKMM21}).
Therefore without loss of generality one can consider
those intersecting sets containing the identity. We will
refer to such sets as {\em basic intersecting sets}.

The investigation of the EKR-property of transitive permutation groups and related concepts is an active topic of research (see \cite{HKMM21,HKMM21-n2,HKKMM21,LSP,Meagher19,MR21, MRS21,MS21,R21,S19}).
This paper initiates a program aimed at obtaining deeper understanding
of transitive permutation groups, not having the EKR-property, 
with small point stabilizers.  
As a starting point we consider groups with
cyclic stabilizers of prime order, in particular 
those isomorphic to $\ZZ_2$ and $\ZZ_3$. 
We remark that in this context the question on
whether their orbitals are connected 
and self-paired or not is essential,
see Section~\ref{sec:pre}.  

The paper is organized as follows. 
In Section~\ref{sec:pre} we gather basic definitions and fix the notation. 
Section~\ref{sec:2} deals with groups having point stabilizers of order $2$
while Section~\ref{sec:p} deals with solvable groups having point stabilizers of odd prime order. In Section~\ref{sec:3}  some special results on intersecting sets in groups with point stabilizer of order $3$ are proved.
In Section~\ref{sec:4/3} a construction of an infinite family of groups with point stabilizer of order $3$ and intersection density $4/3$ is given. Finally, in Section~\ref{sec:ado} infinite families of groups with point stabilizer of order $3$ and arbitrarily large intersection density  are constructed.


\section{Preliminaries}
\label{sec:pre}
\noindent

\subsection{Permutation groups and orbital digraphs}
\noindent

Let $G \leq \Sym(V)$. The {\em orbitals} of $G$ are the orbits of $V \times V$ under the canonical action of $G$ on $V \times V$. 
The set $\Delta^*:=\{ (u,v) \colon (v,u) \in \Delta \}$ is also an orbital, and in the case when $\Delta=\Delta^*$, $\Delta$ is called {\em self-paired}. 
Supppose that $G$ is transitive. Then the set 
$\{(v,v) \colon v \in V\}$ is an orbital, and any orbital distinct from this is 
called {\em non-trivial}. For a non-trivial orbital $\Delta$, let 
$\Gamma_\Delta$ be the digraph with vertex set $V$ and arc set 
$\Delta$, that is, the so-called {\em orbital digraph} 
$\Gamma_\Delta=\Gamma(G,\Delta)$ of $G$ relative to $\Delta$.
 We say that $\Delta$ is {\em connected} if so is the associated 
digraph $\Gamma_\Delta$, that is, for 
any two vertices $u$ and $v$, there is a sequence of vertices $u=w_1,w_2,
\ldots, w_k=v$ such that for every $1 \le  i \le k-1$, $(w_i,w_{i+1}) 
\in \Delta$ or 
$(w_{i+1},w_i) \in \Delta$.  
Using a standard terminology, for  a prime $p$
a $p$-element is an element of order a power of $p$.

\begin{proposition}\label{prop:orbital}
Let $G \le \Sym(V)$ be a transitive permutation group and let 
$\Delta$ be a non-trivial orbital containing the pair $(u,v)$. Then
\begin{enumerate}[(i)]
\itemsep=0pt
\item $\Delta$ is self-paired if and only if there exists a 
$2$-element $g \in G$ such that 
$g(u)=v$ and $g^2 \in G_u$, 
\item $\Delta$ is connected if and only if 
$G=\la G_u, g\ra$ for every $g \in G$ with 
$g(u)=v$.    
\end{enumerate}
\end{proposition}

Let $H$ be a point stabilizer and $\Delta$  an orbital of 
a transitive permutation group $G$.  
Then it is well-known that the orbital digraph 
$\Gamma_\Delta=\Gamma(G,\Delta)$ corresponds 
to the so-called 
{\em double coset graph} 
$\textrm{Cos}(G,H,HgH)$, where $g \in G$ such that $(u,u^g) \in \Delta$ and 
$H=G_u$. In general, for a subset 
$S \subseteq G$, the double coset graph is defined to have vertex set $G/H$ (the set of left cosets of $H$ in $G$) and arc set consisting of pairs $(xH,yH)$, $x, y \in G$ and $x^{-1}y \in HSH$.

We end this subsection with a result about quotient graphs
of arc-transitive graphs which will be needed in Section~\ref{sec:p}.
For terms not defined here we refer the reader to \cite{LP08}
(see also \cite{Lorimer}). 

\begin{proposition}\label{prop:quo}
{\rm\cite[Lemma~2.5]{LP08}}
Let $\Gamma$ be a connected $G$-arc-transitive graph of 
valency $p$ for an odd prime $p$, and let $N \lhd \Aut(\Gamma)$ 
be a semiregular subgroup with at least $3$ orbits. 
Then $\Gamma_N$ is a connected 
$\bar{G}$-arc-transitive graph of valency $p$, where  
$\bar{G}$ is the image of $G$ under its action on $V(\Gamma_N)$.
\end{proposition}


\subsection{Derangement graphs}
\noindent

Let $G \le \Sym(V)$. The fixed-point-free elements of $G$ are also called {\em derangements}. Following \cite{MRS21}, the {\em derangement graph} of $G$ 
is defined to be the Cayley graph $\Gamma_G=\Cay(G,\mathcal{D})$ with vertex set 
$G$ and edge set consisting of all pairs $(g,h) \in G \times G$ such that 
$g^{-1}h \in \mathcal{D}$, where $\mathcal{D}$ is the set of all derangements of $G$.  
Clearly, $\mathcal{D}$ is closed under conjugation by elements of $G$, which shows 
that $\mathrm{Inn}(G) \le \Aut(\Gamma_G)$, where 
$\mathrm{Inn}(G)$ denotes the group 
of all {\em inner automorphisms} of $G$. 
The lemma below gives an additional kind of automorphisms of 
a derangement graph. 

\begin{lemma}\label{lem:inversion map is automorphism}
For every $G \leq \Sym(V)$, the inverse map $\iota: G \to G$ defined by 
$\iota(g)=g^{-1}$ $(g \in G)$ is an automorphism of the derangement graph $\Gamma_G$.
\end{lemma}
\begin{proof}
Suppose that $x$ and $y$ are two non-adjacent vertices of $\Gamma_G$. Then there exists  points $v,u \in V$ such that $x(v)=u=y(v)$. 
Clearly,  $x^{-1}(u)=v=y^{-1}(u)$, implying that $x^{-1}$ and $y^{-1}$ are also non-adjacent in $\Gamma_G$.
\end{proof}

In the terminology of 
derangement graphs, an intersecting set of a transitive
permutation group 
$G$ is an independent set of $\Gamma_G$, or equivalently, 
a clique of its complement $\overline{\Gamma_G}$. 
Therefore, 
$
\rho(G)=\omega(\overline{\Gamma_G})/|G_v|,
$ 
where the {\em clique number} $\omega(\overline{\Gamma_G})$ is the size 
of a maximum clique of $\overline{\Gamma_G}$.   
In the special case with stabilizers of order $2$ or $3$, 
the complement of the derangement graph is   
arc-transitive. 

\begin{lemma}\label{lem:arc-transitivity for small stabilizer}
Let $G$ be a transitive group with point stabilizers of order $2$ or $3$. 
Then the complement $\overline{\Gamma_G}$ is arc-transitive.
\end{lemma}
\begin{proof}
Observe that $\overline{\Gamma_G}=\Cay(G,\S)$ where $\cal{S}$ is the set of all non-identity elements of $G$ that fix a point. As $G$ is transitive, its point stabilizers are conjugate. Thus $\S=C \cup C^{-1}$, where 
$C$ is the conjugacy class of some $g \in G$ fixing a point.  
Using the fact that $\mathrm{Inn}(G) \le \Aut(\overline{\Gamma_G})$ and 
that $\iota \in \Aut(\overline{\Gamma_G})$, we obtain that the stabilizer of the identity element $1=1_G$  of $G$ in $\Aut(\overline{\Gamma_G})$ acts transitively on $\cal{S}$.
Hence $\overline{\Gamma_G}$ is   arc-transitive. 
\end{proof}

For a prime power $q$ with $q \equiv 1 \pmod{4}$, the {\em Paley graph} $P_q$ is defined to be the Cayley graph of $\mathbb{F}_q^+$, 
the additive group of the finite field
$\mathbb{F}_q$ with $q$ elements, 
whose connection set consists of all non-zero squares in $\mathbb{F}_q$. 

\begin{proposition}\label{prop:clique number of Paley}
{\rm \cite[Theorem~1]{BDR88}}
Let $q=p^{2n}$ where $p$ is a prime such that $q \equiv 1 \pmod{4}$.  
Then $\omega(P_q)=\sqrt{q}$.
\end{proposition}

\subsection{Intersection density}
\noindent

In this subsection we list four results from \cite{HKMM21}.

\begin{proposition}\label{pro:regular}
{\rm (\cite[Proposition~2.6]{HKMM21})}
Let $G$ be a transitive permutation group containing a
semiregular subgroup $H$ with $k$ orbits. 
Then $\rho(G) \le k$. In particular, if $H$ is regular then $\rho(G)=1$.
\end{proposition}

\begin{proposition}\label{prop:semi-inv}
{\rm \cite[Proposition~3.1]{HKMM21}}
Let $G$ be a transitive permutation group admitting a semiregular subgroup $H$ whose
orbits form a $G$-invariant partition $\mathcal{B}$, and let $\bar{G}$ be 
the image of $G$ under its action on $\mathcal{B}$. 
Then $\rho(G) \le \rho(\bar{G})$.
\end{proposition}

\begin{proposition}\label{thm:deg-p}
{\rm \cite[Theorem~1.4]{HKMM21}}
For a transitive permutation group $G$ of prime power degree the intersection density
$\rho(G)$ is equal to $1$.
\end{proposition}

We end this subsection with a lemma that can be extracted
from the proof of \cite[Theorem~1.5]{HKMM21}

\begin{lemma}\label{lem:bloki-2}
{\rm \cite{HKMM21}}
Let $p$ be a prime and let $G$ be a transitive permutation group
of degree $2p$. If $G$ admits $G$-invariant partition  $\mathcal{B}$
consisting of two blocks then $\rho(G)=1$.
\end{lemma}


\subsection{Properties of $\PSL(2,q)$}
\noindent

We will need two  well-known results about the group $\PSL(2,q)$.
The first  one can be easily   derived from the structure of subgroups of 
$\PSL(2,q)$ (see, for example, \cite{Sz}), so we omit the proof.
For the second one we refer the reader to \cite{Ma}.

\begin{proposition}\label{prop:3-class}
Let $q=3^n$ and $n > 1$. Then 
$\PSL(2,q)$ has one conjugacy class of subgroups of order $3$ if $n$ is odd, and two otherwise.
\end{proposition}

\begin{proposition}\label{thm:23-gen}
{\rm \cite{Ma}}
Unless  $q=9$ the group $\PSL(2,q)$ 
is generated by an element of order $2$ 
and an element of order $3$.  
\end{proposition}

Finally, 
it is known that if $F$ is an algebraic closed field of characteristic $p$ 
and $A \in \SL(2,F)$ such that $A$ is not in the center $Z=Z(\SL(2,F))$, then $A^3 \in Z$ if and only if $\Tr(A)$ is either $1$ or $-1$ 
(see \cite[(6.19) in Chapter 3]{Sz}).  This implies the following statement.

\begin{proposition}\label{pro:trace}
Let $A \in \SL(2,q)$ be a matrix such that $A \ne \pm I$. Then
$A^3=\pm I$ if and only if its
trace $\Tr(A)$ is either $1$ or $-1$.
\end{proposition}

\begin{proposition}
\label{cor:pgl-psl}
 Let $q=p^k$ for a prime $p$, $p \ne 3$.
Further let $G=\mathrm{PSL}(2,q)$  be
considered in its transitive action on the cosets of a subgroup
isomorphic to $\ZZ_3$. Then a non-canonical  basic intersecting set in $G$ 
contains no point stabilizer as a proper
subset.
\end{proposition}

\begin{proof}
Let us assume that there is an intersecting set $\mathcal{F}$ 
in $G$ containing $I$, $B$, $B^2$ and $C$, 
where $B$ and $C$ are
two elements of order $3$ in $G$ such that $C\not\in\{B,B^2\}$. 

Consider the above two elements $B$ and $C$ 
 in $G$ as `matrices', that is elements in $\mathrm{SL}(2,q)$.
Then the minimal polynomial $m_B(x)$ of $B$ is either
$x^2+x+1$ or $x^2-x+1$, depending on whether $B^3=I$
or $B^3=-I$ in $\mathrm{SL}(2,q)$.
If $m_B(x)=x^2+x+1$ consider the product $C(I+B+B^2)$.
We have
$$
\mathbf{0}=C \cdot \mathbf{0} =C(I+B+B^2)= CI + CB + CB^2.
$$
Consequently, the sum of traces $Tr(C I)$, $Tr(CB)$ and
$Tr(CB^2)$ equals zero. However, by our assumption each of
$C$, $CB$ and $CB^2$ is an element of order $3$, and therefore
its trace (as an `element' of $\mathrm{SL}(2,q)$) is either $1$ or $-1$.
Therefore, since $p\ne 3$, 
the sum of these three traces cannot be equal to zero.
The argument in case when $m_B(x)=x^2-x+1$ is analogous.
We just have to consider the expression $C(I-B+B^2)$, 
completing the proof of Proposition~\ref{cor:pgl-psl}.
\end{proof}



\section{Groups with point stabilizers of order $2$}
\label{sec:2}
\noindent

\begin{lemma}\label{lem:stabilizer Z_2 - comuting elements}
Let $G$ be a transitive group with point stabilizers of order $2$, and let 
$\mathcal{F}$ be a basic  intersecting set. Then $\la \mathcal{F} \ra$ is an elementary abelian $2$-group.
\end{lemma}
\begin{proof}
Let $x, y \in \mathcal{F}$ be distinct and non-identity. 
Since $x$ and $y$ are intersecting with the identity, it follows that each fixes a point, 
hence has order $2$. Moreover, since $x$ and $y$ are intersecting, it follows that $x^{-1}y=xy$ is also of order $2$, which implies that $x$ and $y$ commute.
\end{proof}

In the next proposition we provide some sufficient condition for 
the EKR property of a transitive group with point stabilizers of order $2$.

\begin{proposition}
A transitive group $G$ with point stabilizers of order $2$ has the EKR property if 
\begin{enumerate}[(i)]
\itemsep=0pt
\item $G$ has a cyclic Sylow $2$-subgroup, or 
\item $G$ is a nilpotent group. 
\end{enumerate}
\end{proposition}
\begin{proof}
Part (i) is a direct consequence of 
Lemma~\ref{lem:stabilizer Z_2 - comuting elements}. 

To show part (ii) we proceed by induction on the 
number of prime divisors of $|G|$. 
If $G$ is a $2$-group, then the statement follows from 
Proposition~\ref{thm:deg-p}. 
Assume that $|G|$ is divisible by an odd prime $p$, and let 
$P$ be a Sylow $p$-subgroup of $G$. Then $P \lhd G$ because $G$ is nilpotent. 
As $P$ is also semiregular, Proposition~\ref{prop:semi-inv} 
can be applied with $\mathcal{B}$ being the set of $P$-orbits. 
It follows that $\rho(G) \le \rho(\bar{G})$, where 
$\bar{G}$ is the image of
$G$ under its action on $\mathcal{B}$. 
Moreover, 
$\bar{G}$ is transitive with point stabilizers of 
order $2$, it is also nilpotent with the smaller number of prime divisors
than $|G|$. Hence the induction hypothesis yields 
$\rho(\bar{G})=1$, and (ii) follows. 
\end{proof}

It can be deduced  from Lemma~\ref{lem:stabilizer Z_2 - comuting elements} 
that a maximal intersecting set 
of transitive group with point stabilizer of order $2$
cannot have size $3$. 
Namely, if $\mathcal{F}=\{1,x,y\}$ is an intersection set, then 
it is easy to see that so is $\{1,x,y,xy\}$.
Nonetheless, the example below shows that maximum intersecting sets in transitive groups  with stabilizers of order $2$ can have any size $d \geq 2$ and $d \ne 3$.

\begin{example}
{\rm  
Let $G=E \rtimes Q$ be the semidirect product of the group 
$E=\la e_1,\ldots,e_n \ra \cong \ZZ_2^n$ with the group $Q=\la a, b\ra$, where 
the action of $a$ and $b$ by conjugation on $E$ is defined by 
$$
e_i^a=\begin{cases} e_{i+1} & \text{if}~1 \le i \le n-1 \\ 
e_1 & \text{if }~i=n. 
\end{cases},\quad
e_i^b=\begin{cases} e_ie_n & \text{if}~1 \le i \le n-1 \\ 
e_n & \text{if}~i=n.
\end{cases}
$$
Consider the permutation group  induced by the action of $G$ on 
the cosets of $\la e_1 \ra$.

It is easy to check that the conjugacy class of $G$
containing $e_1$ is equal to the set $C:=\{e_i, e_ie_j : 1 \le i,j \le n\}$. 
This shows that the set $\{1,e_1,\ldots,e_n\}$ is an intersecting set. 
We prove that, if $\mathcal{F}$ is an arbitrary intersecting set, then 
$|\mathcal{F}| \le n+1$. 
We may assume without loss of generality 
that $1 \in \mathcal{F}$, hence $\mathcal{F}\setminus\{1\} \subset C$. 
Write 
$$
\mathcal{F}_1=\mathcal{F} \cap \{e_i : 1 \le i \le n\}~\text{and}~
\mathcal{F}_2=\mathcal{F} \cap \{e_ie_j : 1 \le i,j \le n\}.
$$ 
Note that, from the set theoretic point of view
the sets $\{i,j\}$, $e_ie_j \in \mathcal{F}$, form a classical intersecting family 
of $\{1,\ldots,n\}$, hence $|\mathcal{F}_2| \le n-1$. On the other
hand, it is easy to see that 
$|\mathcal{F}_1| \le 1$ if $|\mathcal{F}_2| \ge 2$. All these yield that 
$|\mathcal{F}| \le n+1$.
}\end{example}

We remark that the group $G$ above
admits no connected orbitals. 
If that was not the case then 
$G=\la e_1, g \ra$ would hold for some element $g \in G$ due to 
Proposition~\ref{prop:orbital}(ii). But then $G=\la E, g \ra=E \rtimes \la g \ra$, so $Q$ 
is a cyclic group, contradicting the fact that $ab \ne ba$.

We conclude this subsection with a connection to character theory.
In fact, the EKR property of a given transitive group $G$ with stabilizers of order $2$ 
can be read off the character table of $G$. 

\begin{proposition}
Let $G$ be a transitive group with point stabiliers of order $2$.  
Then $G$ has the EKR property if and only if 
$$
\sum_{\chi}\frac{ \chi(g)^3}{\chi(1)}=0.
$$
where $\chi$ runs though  the set of all irreducible characters of $G$, and 
$g$ is any non-identity element of $G$ fixing a point. 
\end{proposition}
\begin{proof}
Let $C_1,\ldots,C_r$ be the conjugacy classes of $G$ and let 
$a_{ijk}$ be the class algebra constants of $G$, that is, for $1 \le i,j,k \le r$,
$$
\underline{C}_i\, \underline{C}_j=\sum_{k=1}^{r} a_{ijk} \underline{C}_k
$$
holds in the group algebra $\CC G$ (see \cite[Chapter~30]{JL}). For a subset $X \subseteq G$, 
$\underline{X}$ stands for the element $\sum_{g\in G}a_g g$ with 
$a_g=1$ if $g \in X$, and $a_g=0$ otherwise. Then we have 
\begin{equation}\label{eq:formula}
a_{ijk}=\frac{|G|}{|C_G(g_i)| |C_G(g_j)|}\sum_{\chi }
\frac{ \chi(g_i)\chi(g_j)\overline{\chi(g_k)}}{\chi(1)},
\end{equation}
where $g_i \in C_i, g_j\in C_j$ and $g_k \in C_k$, and 
$\chi$ runs through the set of all irreducible characters of $G$ (see  
\cite[Theorem~30.4]{JL}). 
Now, take $g_1$ to be $g$. It is clear that having 
the EKR property is the same as requiring
$a_{111}=0$, so the proposition follows directly from \eqref{eq:formula}.
\end{proof}


\section{Solvable groups with point stabilizers of odd prime order}
\label{sec:p}
\noindent

We start by observing that solvable transitive permutation groups
with point stabilizers of prime order  and admitting 
a connected self-paired orbital have the EKR property. 
This is a consequence of the following more general result
formulated in a graph-theoretic language.

\begin{proposition}\label{pro:pval}
Let $p$ be a  prime, 
$\Gamma$ be  a connected $p$-valent arc-transitive graph, 
and let $G\le\Aut(\Gamma)$ be
a solvable group acting transitively on arcs of $\Gamma$.  
Then $G$ has the EKR property. 
\end{proposition}

\begin{proof}
Assume first that $p=2$. Then $\Gamma$ is a cycle, and so
$G $ contains a regular subgroup, and thus Proposition~\ref{pro:regular}
implies that $G$ has the EKR property.

Assume from now on that $p$ is an odd prime.
Assume on the contrary that  $\rho(G) >1$
with  $\Gamma$ being the smallest counterexample. In other words, 
if $\tilde{\Gamma}$ is a  connected $p$-valent $\tilde{G}$-arc-transitive graph for 
some solvable group $\tilde{G}$ and $|V(\tilde{\Gamma})| < |V(\Gamma)|$, 
then $\rho(\tilde{G})=1$. 

Let $N$ be a minimal normal subgroup of $G$. Then $N  \cong \ZZ_r^l$ for 
some prime $r$ and $l \ge 1$. If $N$ is transitive, then it is regular, but this 
contradicts the fact that $\rho(G) > 1$ in view of 
Proposition~\ref{prop:semi-inv}. 

If $N$ has more than $2$ orbits, by Proposition~\ref{prop:quo}, 
$\Gamma_N$ is a connected $p$-valent $\bar{G}$-arc-transitive graph with 
$\bar{G} \cong G/N$, where $\bar{G}$ is the image 
of  $G$ under its action on the orbits of $N$. 
The group $\bar{G}$ is solvable, hence by the minimality of 
$\Gamma$, $\rho(\bar{G})=1$. On the other  hand, 
by Proposition~\ref{prop:semi-inv}, $1 < \rho(G) \le \rho(\bar{G})$, 
a contradiction. 

It remains to consider the case when $N$ has two orbits. 
If $N$ is semiregular, then the image of the action of $G$ on the two $N$-orbits 
is $S_2$, hence by Proposition~\ref{prop:semi-inv}, $1 < \rho(G) \le \rho(S_2)=1$, again
a contradiction. If $N$ is not semiregular, then it is easy to show that $r=p$, $N \cong \ZZ_p^2$ and $\Gamma=K_{p,p}$. Clearly the bipartition sets of 
$\Gamma$ form a block system for $G$. 
It follows then from Lemma~\ref{lem:bloki-2} that $\rho(G)=1$, 
a contradition, completing the proof of Proposition~\ref{pro:pval}.  
\end{proof}

In the example below we show that the above 
proposition does not extend to 
solvable transitive permutation groups 
with the point stabilizer of order $p$ but without connected
self-paired orbitals. 
These groups act on
edge-transitive graphs,   first 
studied by Praeger and Xu~\cite{PX}.

\begin{example}{\rm
Let $G=\mathrm{AGL}(1,q)$, where $q=p^e$, $p$ an odd prime and $e >1$. 
Then $G$ consists of the linear transformations $x \mapsto ax+b$, 
where $a \in \mathbb{F}_q^*$ and $b \in \mathbb{F}_q$.  
Fix non-zero elements  $a, b \in \mathbb{F}_q$ 
such that $a$ is a primitive element, and let 
$\sigma : x \mapsto x+b$ and $\tau : x \mapsto ax$. 
Note that $\la \sigma,\tau \ra=G$.
Then let
$\Gamma$ be the double coset graph
$\mathrm{Cos}(G,H,H\{\tau,\tau^{-1}\}H)$,
where $H=\la \sigma \ra$. 
It follows from basic properties of 
double coset graphs (see, 
for example, \cite[Lemma~2.1]{LLM}) 
that $\Gamma$ is a connected $G$-edge-transitive graph. 
It is easy to see that $H\tau H \ne H\tau^{-1}H$
(the corresponding orbitals are non-self-paired), hence, by 
\cite[Lemma~2.4]{LLM}, $\Gamma$ has valency $2|H|/|H \cap H^\tau|=2p$. 
The conjugacy class of $G$ containing 
$\sigma$ consists of all transformations $x \mapsto x+u$, $u \ne 0$. 
Consequently, these transformations together with the identity
form an intersecting set, 
and so $\rho(G) > 1$. 
Note also that a self-paired orbital of the action of
$G$ on $H=\la \sigma \ra$ is necessarily disconnected. 
}\end{example}


\section{Intersecting sets in groups with point stabilizers of order $3$}
\label{sec:3}
\noindent

In view of Proposition~\ref{pro:pval}, in order to construct 
transitive permutation groups having point stabilizers of prime order $p$
admitting a connected self-paired orbital
and without EKR-property,  nonsolvable groups 
need to be brought  into  consideration. 
Hereafter, we restrict ourselves to  $p=3$.
The lemma below gives a lower bound on the cardinality
of a maximal intersecting set containing a  point stabilizer
 in a transitive group
with point stabilizer of order $3$. 

\begin{lemma}\label{lem:ademir}
Let $G$ be a transitive permutation group. Let $H,K\leq G$ such that $H\cup K$ is an intersecting set. Then $HK$ is also an intersecting set.
\end{lemma}
\begin{proof}
Since $H\cup K$ is an intersecting set, it follows that $hk$ and $kh$ fixes a point for each $h\in H$ and $k\in K$. Hence $1$ is adjacent with every element of $HK$ in $\overline{\Gamma_{G}}$. Let $hk$, $h_1k_1$ be arbitrary in $HK$. By $\sim$ we denote the adjacency relation in $\overline{\Gamma_{G}}$. Then we have 
\begin{align*}
hk\sim h_1k_1  & \Leftrightarrow k \sim h^{-1}h_1k_1 \textrm{ (Using left multiplicaiton by } h^{-1} \textrm{which is an automorphism}) \\
& \Leftrightarrow k^{-1} \sim k_1^{-1}h_1^{-1}h \textrm{ (Using the inversion automorphism})\\
& \Leftrightarrow 1 \sim kk_1^{-1}h_1^{-1}h \textrm{ (Using left multiplicaiton by } k)\\
& \Leftrightarrow 1 \sim k'h'.
\end{align*}
Observe that $k'h'$ is adjacent with $1$ in  $\overline{\Gamma_{G}}$ for every $h'\in H$ and every $k'\in K$ since $H\cup K$ is an intersecting set in $G$. This shows that $HK$ is a clique in  $\overline{\Gamma_{G}}$.
\end{proof}

\begin{proposition}
\label{lem:2}
Let $G$ be a transitive group acting on a set $V$,
with point stabilizer $H=\la x \ra \cong \ZZ_3$
of order $3$, and admitting
a maximal  intersecting set $\mathcal{F}$  
containing $H$ as a proper subset. Then 
\begin{enumerate}[(i)]
\itemsep=0pt
\item $|\mathcal{F}|\ge 9$.
\item For every $y \in \mathcal{F} \setminus H$, 
the group $\la x, y \ra $ is  isomorphic either to  $\ZZ_3^2$ or 
to the unique non-abelian group of order $27$ with exponent $3$. 
\end{enumerate}
\end{proposition}
\begin{proof}
By assumption there exists  $y\in \mathcal{F}\setminus H$.
It follows that 
$\{1,x,x^2,y\}$ is a clique in $\overline{\Gamma_G}$. 
Since, by Lemma~\ref{lem:inversion map is automorphism},
the mapping $\iota : g \mapsto g^{-1}$ 
($g \in G$) is an automorphism of $\overline{\Gamma_G}$, it follows
that $\{1,x,x^2,y,y^2=y^{-1}\}$ is an intersecting set of $G$,
and so (i) follows by Lemma~\ref{lem:ademir}.

As for part  (ii), let $P=\la x, y \ra$. 
If $z=[x,y]=1$, then  clearly $P \cong \ZZ_3^2$. Hence assume that 
$z  \ne 1$.  We show that $[x,z]=[y,z]=1$. 
It follows from $(x^2y)^3=1$ that $y^2xy^2=x^2yx^2$ and from 
$(xy)^3=1$ that $x^2y^2x^2=yxy$ and $(yx)^2=x^2y^2$.  
Using these we compute 
$$
z^{x}=x [x,y] x^2=y^2xyx^2=(y^2xy^2)y^2x^2=x^2y(x^2y^2x^2)=x^2y^2xy=z, 
$$
$$
z^{y}=y [x,y] y^2=yx^2y^2x=yx(xy^2x)=yxyx^2y=(yx)^2xy=x^2y^2xy=z.
$$
Since $xz=y^2xy$ is a conjugate of $x$ 
it has order $3$, and since $[x,z]=1$, we conclude that so does $z$. 
It follows that $\la z \ra \cong \ZZ_3$, $P/\la z \ra \cong \ZZ_3^2$, and therefore, 
$P$ is isomorphic to the unique non-abelian group of order $27$ with exponent $3$.
\end{proof}

\begin{corollary}
\label{cor:1}
Let $G$ be a transitive group acting on a set $V$,
with point stabilizer 
of order $3$, and admitting 
a maximal basic intersecting set $\mathcal{F}$  of $G$ such that
 $|\mathcal{F}|=4$.
Then no point stabilizer $G_v$, $v\in V$,  is   contained in 
$\mathcal{F}$. In other words, $\mathcal{F}=\{1,x,y,z\}$
with $x$, $y$ and $z$ belonging to different point stabilizers.
\end{corollary}

\begin{example}{\rm
There exists a connected cubic symmetric graph of order $720$ with automorphism group $A$ of order $2160$. 
Using MAGMA \cite{Mag}, one can obtain that the group $A$ is isomorphic to $(3\cdot A_6)\rtimes \ZZ_2$, where $3\cdot A_6$ denotes the triple cover of $A_6$. Let $P$ be a Sylow $3$-subgroup of $A$. Then $P\cong \langle a,b,c\mid a^3=b^3=c^3=1,[a,b]=c, [a,c]=1,[b,c]=1 \rangle$ is the unique non-abelian group of order $27$ with exponent $3$. 
Again, we established with MAGMA that maximum intersecting sets of $A$ are of size $9$ and furthermore every maximum intersecting set of $A$ containing identity is contained in $P$ or one of its $10$ conjugates in $A$. 
All non-central elements of order $3$ from $P$ are conjugate in $A$, and the subgraph of the derangement graph of $A$ induced by these $24$ non-central elements of $P$ is a disjoint union of 8 triangles of the form $\{x,xc,xc^2\}$, with $x$ a non-central element of $P$. Then the maximum intersecting sets containing $1$ that lie inside $P$ are just the transversals  of $\langle c \rangle=Z(P)$ in $P$. Observe that there are $3^8$ such transversals. 
Consequently,  all possibilities between $0$ and $4$ full vertex-stabilizers, contained in such a maximum intersecting set of $A$, can occur.  
}\end{example}





\section{Groups with point stabilizer of order $3$ and intersection density $4/3$}\label{sec:4/3}
\noindent

In this section we give an infinite family of transitive permutation groups
with  intersection density $4/3$. These groups arise 
from the groups $\PSL(2,q)$ 
acting on cosets of a subgroup of order $3$.   
In view of Proposition~\ref{cor:pgl-psl}, 
the corresponding non-canonical basic intersecting sets
are of the form $\{1, x,y,z\}$ where $x$, $y$ and $z$ 
are from different point 
stabilizers.

\begin{theorem}\label{the:istvan}
Let $p$ be a prime, let $q=p^e \equiv 1 \pmod 3$, 
and let  $G$ be a transitive permutation group
arising from $\PSL(2,q)$ acting on cosets of a subgroup of order $3$. 
Then 
\begin{eqnarray} \label{eq:rho}
\rho(G)&=&\begin{cases} 4/3 & \text{if } p \ne 5, \\ 
2 & \text{if } p=5. 
\end{cases}
\end{eqnarray}
\end{theorem}

\begin{proof}
Observe that all cyclic subgroups of order $3$ are conjugate in 
$\PSL(2,q)$, while the elements of order $3$ form 
two conjugacy classes, with two elements 
$x$ and $x^{-1}$ of order $3$ belonging to 
different conjugacy classes. In  $\PGL(2,q)$, however,
these two conjugacy classes merge into a single class. 
Consequently, every element of order $3$ in
$\PSL(2,q)$ belongs to a point stabilizer (with regards to the
action considered in this theorem).

Let $\Gamma$ be the subgraph of the complement of the
derangement graph 
$\Gamma_G$ of 
$G$ induced by the set of neighbours 
of the identity element of $G$. In other words, the vertices of $\Gamma$
are the elements of order $3$, with $h ,g$  adjacent if and only if 
$h^{-1} g$ is of order $3$. 
Since a clique in $\Gamma$ together with the identity element
gives rise to an intersecting set for $G$, the statement (\ref{eq:rho})
will be proved if we show that
\begin{equation}\label{eq:omega}
\omega(\Gamma)=
\begin{cases} 3 & \text{if } p \ne 5, \\ 
5 & \text{if } p=5,
\end{cases}
\end{equation}
where $\omega(\Gamma)$ denotes the clique number of $\Gamma$.

As before we will think of elements of $\PGL(2,q)$ 
as `matrices', that is, as elements of $\GL(2,q)$. 
Note that, $A, B \in \SL(2,q)$ represent the same 
element in $\PSL(2,q)$ if and only if $B=-A$. 

In view of the observation made in the first paragraph,
the action of  $\PGL(2,q)$  on $V(\Gamma)$ by conjugation
is transitive, and the 
corresponding image is therefore a transitive subgroup of $\Aut(\Gamma)$. 
In particular, every vertex of $\Gamma$ is contained in a maximum clique. 
Let us fix the vertex 
$$
A_0=\mm{r}{0}{0}{r^2}, 
$$ 
where $r^3=1$ and $r \ne 1$.
The existence of such an $r$ is guaranteed 
by the condition $q \equiv 1 \pmod 3$.
Let $\mathcal{K}$ be a maximum clique containing $A_0$. 
The vertex stabilizer of $A_0$ in $\PGL(2,q)$ coincides with the centralizer 
$\Lambda=\{ L_a : a \in \mathbb{F}_q^* \}$ of 
$A_0$ in $\PGL(2,q)$, isomorphic to $\ZZ_{q-1}$, where 
$$
L_a:=\mm{a}{0}{0}{1}, \ a \in \mathbb{F}_q^*.
$$

We now show that the neighborhood  
$\Gamma(A_0) $ splits into the following  
$\Lambda$-orbits: 
$\{A_0^{-1}\},~O_1,~O_1'$ and $O_2$ if $p \ne 2$, and 
$\{A_0^{-1}\},~O_1$ and $O_1'$ if $p=2$, where  
\begin{eqnarray*}
O_1 &=& \big\{ A_x=\mm{-r^2}{x}{0}{-r} : x \in \mathbb{F}_q^* \big\}, \\
O_1' &=& \big\{ A_x^T : x \in \mathbb{F}_q^* \big\}, \\ 
O_2 &=& \Big\{ B_x=\mm{\frac{r}{r-1}}{-2x}{\frac{1}{3x}}{\frac{-1}{r-1}}  : x \in \mathbb{F}_q^*  \Big\}.
\end{eqnarray*}  
It is clear that $A_0$ and $A_0^{-1}$ are adjacent in $\Gamma$. 
In view of Proposition~\ref{pro:trace}, we may choose
$A \in \Gamma(A_0)\setminus\{A_0^{-1}\}$ with
$\Tr(A)=1$. 
Then $A=\mm{a}{b}{c}{d}$ for some $a, b, c, d \in \mathbb{F}_q$ with  
$ad-bc=1$ and $a+d=1$. 
Also, $\Tr(A_0^{-1}A)=r^2a+rd=\pm 1$. 

Suppose first that $p \ne 2$. 
Using the fact  that $a+d=1$, we have $(r^2-r)a=\pm 1 - r$. 
If $(r^2-r)a=1-r$, then we have $a=-r^2$, $d=-r$, and $bc=ad-1=0$. 
If $c=0$, then $A \in O_1$, and if $b=0$, then $A \in O_1'$.
Suppose now that  $(r^2-r)a=-1-r=r^2$. 
Then we have that $a=r/(r-1)$, $d=-1/(r-1)$, and 
$bc=ad-1=-r/(r-1)^2-1=-r/(-3r)-1=-2/3$. It follows that $A \in O_2$.
We have therefore shown that 
$\Gamma(A_0)=\{A_0^{-1}\} \cup O_1 \cup O_1' \cup O_2$. 

Suppose now that $p=2$. Then  $r^2a+rd=1$,
and since $d=1-a$ it follows that 
$a=(r+r^2)a=r+1=r^2$, $d=r$ and $bc=0$. We obtain that 
$\Gamma(A_0)=\{A_0^{-1}\} \cup O_1 \cup O_1'$. 

Straightforward computations show that, for every $x \in \mathbb{F}_q^*$, 
$$
A_x= L_x A_1 L_x^{-1},~A_x^T=L_x^{-1} A_1^T L_x,~\text{and}~ 
B_x=L_x B_1 L_x^{-1}. 
$$
Thus each of $O_1$, $O_1'$ and $O_2$ is a $\Lambda$-orbit, as claimed.   
It can be directly checked that for every $x \in \mathbb{F}_q^*$
$$ 
A_x^{-1}=\begin{bmatrix} -r & -x \\ 0 & -r^2\end{bmatrix}~\text{and}~
B_x^{-1}=\begin{bmatrix} \frac{r+2}{3} & 2x \\ -\frac{1}{3x} & 
\frac{2r+1}{3r} \end{bmatrix}. 
$$ 
Using these, the traces $\Tr(XY)$ for 
$X \in \{A_x^{-1},(A_x^T)^{-1},B_x^{-1}\}$ and $Y \in \{A_y,A_y^T,B_y\}$ can 
be computed directly. The results are collected in Table~1.
In particular, $A_1$ and $A_1^T$ are adjacent, and so 
$|\mathcal{K}| \ge 3$. 
Combining this with Proposition~\ref{cor:pgl-psl}
we have that $A_0^{-1} \notin \mathcal{K}$.

\renewcommand{\arraystretch}{1.8}
\begin{table}[t]
\begin{center}
\begin{tabular}{|c|c|c|c|} \hline
                      & $A_y$ & $A_y^T$ & $B_y$ \\ \hline
$A_x^{-1}$        & $2$    &   $2-xy$    & $-\frac{x}{3y}$ \\ \hline
$(A_x^T)^{-1}$ & $2-xy$    &   $2$    & $2xy$ \\ \hline
$B_x^{-1}$        &  $-\frac{y}{3x}$    &  $2xy$    &  $\frac{2}{3}\Big(1+\frac{x}{y}+\frac{y}{x}\Big)$ 
\\ \hline
\end{tabular}
\\ [+1.5ex]
\caption{The traces $Tr(XY)$, $X \in \{A_x^{-1},(A_x^T)^{-1},B_x^{-1}\}$ and 
$Y \in \{A_y,A_y^T,B_y\}$.}
\end{center}
\end{table}

Assume for the moment that $\mathcal{K} \setminus \{A_0\}$ is not contained 
in $O_2$. It is easy to show that the mapping $\tau: X \mapsto X^T$, 
$X \in V(\Gamma)$, is an automorphism of $\Gamma$. This 
automorphism fixes $A_0$ and $A_0^{-1}$, swaps $O_1$ with $O_1'$, and maps $O_2$ onto itself. Thus $\la \Lambda, \tau \ra$ is transitive 
on $O_1 \cup O_1'$, and so we may, without loss of generality, 
assume that $A_1 \in \mathcal{K}$. 
Now, if $A \in \mathcal{K}\setminus\{A_0, A_1\}$ then 
$\Tr(A_1^{-1}A)=\pm 1$. 
Using Table~1, we find 
$A \in \{A_1^T,A_3^T ,B_{1/3},B_{-1/3} \}$ if $p \ne 2$, and $A=A_1^T$ if $p=2$. 
This shows that $\mathcal{K}=\{A_0,A_1,A_1^T\}$ if $p=2$. 
Let $p \ne 2$. If $p \ne 5$, then using  Table~1 again, 
we obtain that no two vertices in the set 
$\{A_1^T,A_3^T ,B_{1/3},B_{-1/3} \}$ are adjacent, and hence 
$|\mathcal{K}|=3$. If $p=5$, then a direct computation yields   
$\mathcal{K}=\{A_0,A_1,A_1^T,B_{1/3},B_{-1/3}\}$, 
and so $|\mathcal{K}|=5$.

Observe that the above arguments show that 
 $\omega(\Gamma)=3$ for $p=2$, as 
claimed in \eqref{eq:omega}.
It remains to consider the case $p \ne 2$ and 
$\mathcal{K} \setminus \{A_0\} \subseteq O_2$. As $O_2$ is a $\Lambda$-orbit, we may assume that $B_1 \in \mathcal{K}$. 
Also, assume 
that 
$B_z \in \mathcal{K}$ for 
$z \ne 1$.
Then  $(\Tr(B_1^{-1}B_z)-1)
(\Tr(B_1^{-1}B_z)+1)=0$ holds.  
Using Table~1, we find 
that $z$ is a root of the polynomial  
\begin{eqnarray}\label{eq:poli}
f(x)=
\Big(x^2-\frac{1}{2}x+1\Big)\Big(x+2\Big)\Big(x+\frac{1}{2}\Big).
\end{eqnarray}

If $p=5$ then $x^2-\frac{1}{2}x+1=(x+1)^2$, $x+\frac{1}{2}=x+3$,
and so $z\in\{2,3,4\}$.
Consequently, $\mathcal{K}=\{A_0,B_1,B_2,B_3,B_4\}$,
$|\mathcal{K}| = 5$, and 
$\omega(\Gamma)=5$, as claimed in \eqref{eq:omega}. 
Suppose now that $p \ne 5$ and that there exists a root $w$ of
$f(x)$, $w\notin\{1, z\}$, such that $B_w\in\mathcal{K}$.
Then, by Table~1, we have
\begin{equation}\label{eq:a/b}
\frac{2}{3}\Big(1+\frac{z}{w}+\frac{w}{z}\Big)=\pm 1. 
\end{equation}
We complete the proof of Theorem~\ref{the:istvan}
by showing that (\ref{eq:a/b}) together with the fact that 
$z$ and $w$ are roots of the polynomial $f(x)$
given in (\ref{eq:poli})
leads to a contradiction. We go through all possibilities for $z$ and $w$.

If $z=-2$ and $w=-1/2$ then \eqref{eq:a/b} yields $15=0$ or $27=0$ in 
$\mathbb{F}_q$, a contradiction. 

If $z=-2$ and $w^2-1/2w+1=0$ then  \eqref{eq:a/b} implies
$2-4/w-w=\pm 3$, and so 
$w=-2$ or $w=2/3$. Hence, $6=0$ or $10=0$, a contradiction. 

If $z=-1/2$ and $w^2-1/2w+1=0$ then  \eqref{eq:a/b} implies
$2-1/w-4w=\pm 3$, and so 
$w=1$ or $w=-1$. Hence  $3=0$ or $5=0$, a contradiction. 

Finally, if $z, w$ are both roots of $x^2-1/2x+1$
then $zw=1$ and $z+w=1/2$, and so 
$1+z^2+w^2=-3/4$. Dividing both sides with $zw=1$, we have 
$1+z/w+w/z=-3/4$. 
Combining this with \eqref{eq:a/b}, we conclude that
$-3/4 =\pm 3/2$, a contradiction, completing the proof of 
Theorem~\ref{the:istvan}.
\end{proof}


\section{Groups with point stabilizer of order $3$  and large intersection density}\label{sec:ado}
\noindent

In the example below, using an action of the symmetric group $S_n$, $n\ge 4$, we show that maximum intersecting sets in transitive groups with point 
stabilizers of order $3$ can be arbitrarily large.

\begin{example}{\rm
Let  $n \ge 4$. First observe that for two $3$-cycles
$x=(i\, j \, k), y \in S_n$ we have  that 
$x^{-1}y$ is also a $3$-cycle if and only if 
$y=(i\,k\,j)$ or 
$
y \in \{\, (i\,j\,l),~(j\,k\,l),~(k\,i\,l) \colon 1 \le l \le n-1, l \notin \{i,j,k\}\, \}.
$

Let $G$ be the permutation group induced by the action of $S_n$ on cosets of $\la (1\, 2\, 3) \ra$. We show   that a maximum intersecting set 
of $G$ has size $n-1$. 
Let $\Gamma$ be the complement of the derangement graph of $G$. Observe that the set of neighbours of $id$ in $\Gamma$ is the set of all $3$-cycles from $S_n$. Let $\Gamma_1$ be the subgraph of $\Gamma$ induced by the set of
 common neighbours of $id$ and $(1\,2\,3)$. 
In view of the first paragraph  it follows that 
$V(\Gamma_1)=\{(1\,3\,2\} \cup \{(1\,2\,i),(2\,3\,i),(3\,1\,i):
i \in \{4,5,\ldots,n\}\}$. Furthermore, it follows that the set 
$\{(1\,2\,i):i \in \{4,5,\ldots,n\}\}$ induces a clique in 
$\Gamma_1$. 
Analogously, the sets $\{(2\,3\,i):i \in \{4,5,\ldots,n\}\}$ 
and $\{(3\,1\,i):i \in \{4,5,\ldots,n\}\}$ also induce cliques. 
Moreover, there are no other edges in $\Gamma_1$, that is, $\Gamma_1$ is a 
disjoint union of an isolated vertex (corresponding to $(1\,3\,2)$) and three 
cliques of size $n-3$. Consequently, $\omega(\Gamma_1)=n-3$. 
A clique in $\Gamma$ is therefore obtained from $id$, $(1\, 2\, 3)$
and any of these three maximum cliques in $\Gamma_1$.
In view of arc-transitivity of $\Gamma$, it follows that 
$\omega(\Gamma)=2+n-3=n-1$.
Hence,   $\rho(G)=\omega(\Gamma)/3=(n-1)/3$.
}\end{example}

The simplicity of this construction is a consequence of the fact that
the existence of a connected self-paired orbital
was not required
(see the first paragraph of Section~\ref{sec:p}). With this additional restriction
the problem of finding such groups becomes somewhat more complicated. In the theorem below we give such a construction
using the group $\PSL(2,3^n)$ for $n\ge 3$.
The group $\PSL(2,9)$ is not considered here because
none of its orbitals in the action on cosets of a subgroup of order $3$
is self-paired and connected.  This is a direct consequence of the fact that, by Proposition~\ref{thm:23-gen},
$\PSL(2,9)$ cannot be generated by an involution and an element
of order $3$. However, by Proposition~\ref{thm:23-gen},
for $n\ge 3$ the group $\PSL(2,3^n)$
is generated by an involution and an element of order $3$, 
and thus the theorem below
indeed gives transitive permutation groups with point
stabilizer of order $3$ having connected self-paired orbital.

\begin{theorem}\label{the:ado}
Let $G$ be a transitive permutation group
arising from $\PSL(2,3^n)$, $n\geq 3$, 
acting on cosets of a subgroup of order $3$. Then 
\begin{eqnarray} \label{eq:rho-ado}
\rho(G)&=&\begin{cases} 3^{n-1} & \text{if $n$ is odd}, \\ 
3^{n/2-1} & \text{if $n$ is even}. 
\end{cases}
\end{eqnarray}
\end{theorem}

\begin{proof}
Consider the subgroup $K$ of $\PSL(2,3^n)$ consisting 
of matrices 
$$
M_x=\mm{1}{x}{0}{1}, \ x \in \mathbb{F}_{3^n},
$$ 
(where a matrix is a `matrix' as in the proof of Theorem~\ref{the:istvan}).  

Observe that $K\cong \ZZ_3^n$  is a Sylow $3$-subgroup of $\PSL(2,3^n)$.  
Observe also that $K$ consists precisely of those elements  of $\PSL(2,3^n)$ fixing the subspace $\la [1,0]^T \ra$. There are $3^n+1$ conjugates of $K$, each conjugate corresponding to a Sylow $3$-subgroup of the stabilizer of some projective point.
In what follows, we shall consider the action of $\PSL(2,3^n)$ on the projective line, whose points are the one dimensional subspaces $\la [a,b]^T \ra$, $a, b \in \mathbb{F}_q$.

Suppose first that $n$ is odd. Then, by Proposition~\ref{prop:3-class} any two subgroups of order $3$ in $\PSL(2,3^n)$ are conjugate.
Let $\Gamma$ be the subgraph of the complement of the derangement graph of $G$ induced by the neighbours of the identity matrix $I\in G$. 
The vertex set $V(\Gamma)$ can be partitioned into $3^n+1$ cliques,
each of size $3^n-1$,  corresponding to 
$K\setminus \{I\}$ and its conjugates. 

Consider a neighbour of $M_x$ outside of $K$. (Note that,
since Sylow 3-subgroups of $\PSL(2,3^n)$ and $\PGL(2,3^n)$ are of the same order, all elements of order $3$ are contained in $\PSL(2,3^n)$.) We can assume that this neighbour is of the form $A=\mm{a}{b}{c}{1-a}$. 
Using the assumption that $M_x^2A$ is an elements of order $3$, it follows that its trace must be equal to $\pm 1$. 
If $2cx+1=1$, it follows that $c=0$, but this would imply that $A$ is in $K$, a contradiction. Hence $2cx+1=-1$, and we obtain that $c=-1/x$. 
This implies that all neighbours of $M_x$ outside of $K$ have the 
 entry in the first column and second row equal to $-1/x$. Consequently, no two different elements of $K$ have a common neighbour outside of $K$. This shows that given an edge of $K$, the maximum clique that contains this edge is the one induced by $K$. 
Hence, the maximum clique of $\Gamma$ can be the one induced by $K$, or it can consist of at most one element from each of the conjugates of $K$. Since there are $3^n+1$ such conjugates, it follows that $\omega(\Gamma)\leq 3^n+1$. Adding the identity, which is not in $\Gamma$,
it follows that $\rho(G)\leq (3^n+2)/3$.  

Let  $q=3^n$, and let  $K=K_0, K_1, \ldots, K_q$ be
 the conjugates of $K_0$, where
we may assume that $K_1$ consists of lower triangular matrices with 1 on diagonal fixing the projective point corresponding to $\la [0,1]^T \ra$. 
Suppose that there exists a maximal clique $\mathcal{K}$ of $\Gamma$
not contained in $K_i$ for any $i$ and such that $|\mathcal{K}|\ge 3^n$.
By the argument from the first paragraph, we know that $|\mathcal{K}\cap K_i| \leq 1$, for each $i$, and so we may assume that $\mathcal{K}\cap K_i=\emptyset$ for at most one $i\in \{0,1,\ldots,q\}$. Since   
$\PSL(2,q)$ acts $2$-transitively on projective points, we may assume,
without loss of generality, that $\mathcal{K}\cap K_0\ne \emptyset$
and that $\mathcal{K}\cap K_1\ne \emptyset$.

Let $A_0=M_x$. 
By calculation we see that the only neighbour of $A_0$ 
in $K_1$ is $A_1=\mm{1}{0}{ 1/x}{1}$.
Let $A=\mm{a}{b}{c}{1-a}$ be a common neighbour of $A_0$ and $A_1$. 
Observe that $\Tr(A_0^2A)=2cx+1$ and $\Tr(A_1^2A)=\frac{2b}{x}+1$.
Using the assumption that $A_0^2A$ and $A_1^2A$ are elements of order $3$, it follows that their traces must be equal to $\pm 1$, implying that $c=-1/x$ and $b=-x$.
Hence   \[
A=
  \begin{bmatrix}
    a & -x\\
    \frac{-1}{x} & 1-a
  \end{bmatrix}.
\]
Since $det(A)=a-a^2-1$, and since $A$ is in $\PSL(2,q)$, it follows that $a-a^2-1=\pm 1$.
There are at most $4$ elements $a$ in 
$\mathbb{F}_q$ that satisfy this equation. 
It follows that for a fixed $x\in \mathbb{F}_q^*$ there are at most four possibilities
for the matrix $A$, and therefore 
$A_0$ and $A_1$ can have at most 4 common neighbours, that is the clique containing the edge $\{A_0,A_1\}$ can have size at most $6$. Since $3^n-1>6$, it follows that $\omega(\Gamma)=3^n-1$. 
This shows that $\rho(G)=3^n/3=3^{n-1}$.

Suppose now that $n$ is even. The elements $M_x$ of $K$ split into two conjugacy classes, depending on whether $x$ is a square in 
$\mathbb{F}_{3^n}^*$
or a non-square.
It is easy to see that the subgraph of $\overline{\Gamma_G}$ 
induced by $K$ is isomorphic to the Paley graph or its complement. 
Combining Proposition~\ref{prop:clique number of Paley}
with the fact that the Paley graph and its complement  are isomorphic  
we obtain that a maximum clique of the latter graph
$\overline{\Gamma_G}[K]$ is of size $3^{n/2}$.
Using the same argument as for $n$ odd we obtain that a clique in 
$\overline{\Gamma_G}$
that is not contained inside 
$K_i$, $i\in\{0,1,\ldots,q\}$ can have size at most $6$. Since $n\geq 4$, it follows that $3^{n/2}\geq 9$, and so the clique number of 
$\overline{\Gamma_G}$ is $3^{n/2}$, implying that $\rho(G)=3^{n/2-1}$. 
\end{proof}

As a final remark regarding possible future work, we would like to 
pose the following problem which generalizes the 
results of this paper for cyclic groups of orders $2$
and $3$. 
 
\begin{problem}
Given a group $H$ determine all possible intersection densities 
of transitive permutation groups having point stabilizers 
isomorphic to $H$.
\end{problem}

 \end{document}